\documentclass[12pt,a4paper,oneside]{amsart}

\usepackage{amsfonts, amsmath, amssymb, amsthm, amscd}
\usepackage{hyperref}
\hypersetup{
  colorlinks   = true, %Colours links instead of ugly boxes
  urlcolor     = blue, %Colour for external hyperlinks
  linkcolor    = blue, %Colour of internal links
  citecolor   = red %Colour of citations
}
\usepackage{anysize}
\marginsize{2.2cm}{2.2cm}{2.2cm}{2.2cm}

\usepackage[pdftex]{graphicx}

\newtheorem{theorem}{Theorem}

\theoremstyle{definition}

\theoremstyle{remark}

\newcounter{fig}
\newcommand{\f}{\refstepcounter{fig} Fig. \arabic{fig}. }

\newcommand{\const}{\mathop{\mathrm{const}}}

\title{$3$-webs generated by confocal conics and circles}
\author{Arseniy~Akopyan}
\address{Arseniy~Akopyan, Institute of Science and Technology Austria (IST Austria), Am Campus~1, 3400 Klosterneuburg, Austria}
\email{akopjan@gmail.com}

\thanks{Supported by People Programme (Marie Curie Actions) of the European Union's Seventh Framework Programme (FP7/2007-2013) under REA grant agreement n$^\circ$[291734].}

\begin{document}

%\kewyword{3-webs, confocal conics, Apollonian circles}

\begin{abstract}

We consider families of confocal conics and two pencils of Apollonian circles having the same foci.
We will show that these families of curves generate \emph{trivial $3$-webs} and find the exact formulas describing them.
\end{abstract}

\maketitle

\section*{Introduction}

The concept of \emph{webs} was invented by W.~Blaschke and connected with many parts of Geometry \cite{blaschke1955einfuhrung}.
Let us recall, that a \emph{trivial $3$-web} in a planar domain $\Omega$ are three families of smooth curves such that there is a diffeomorphism $\varphi:\Omega \rightarrow \Omega' \subset \mathbb{R}^2$ taking the families to sets of lines parallel to the sides of a fixed triangle. See \cite[Lecture 18]{fuchs2007mathematical} as an introduction to the topic. 
There are several non-trivial $3$-webs formed by line- and circle-families, see \cite{nilov2014onnewconstructions} for history of the problem and new examples of such webs.

In this paper we consider webs in the positive quadrant $\mathbb{R}^2_+$ formed by confocal conics and pencils of\emph{ Apollonian circles}, which are defined in the following way:
Let $F_1$ and $F_2$ be two points in the plane called \emph{foci}.
The family of circles passing through $F_1$ and $F_2$ is called an \emph{elliptic Apollonian pencil}.
The pencil of circles orthogonal to all the circles of the first family is called a \emph{hyperbolic Apollonian pencil}.
Each circle from the latter family is a locus of points~$X$ such that $|XF_1|/|XF_2|=\const$.
We refer to \cite{akopyan2007geometry, glaeser2016universe} for this and other classical results related with circles and conics.

We construct four trivial webs defined in the positive quadrant $\mathbb{R}^2_+$. 
These webs consist of the following families of curves:\\
\emph{
1) Both families of Apollonian circles and confocal hyperbolas with foci $F_1$ and $F_2$;\\
2) Both families of Apollonian circles and confocal ellipses with foci $F_1$ and $F_2$;\\
3) Families of confocal ellipses, confocal hyperbolas and the hyperbolic Appolonian pencil with foci $F_1$ and $F_2$;\\
4) Families of confocal ellipses, confocal hyperbolas and the elliptic Appolonian pencil with foci $F_1$ and $F_2$.
}

We prove that these webs are trivial by showing a diffeomorphism from a domain $\Omega \subset \mathbb{R}^2$ to the positive quadrant $\mathbb{R}^2_+$ which maps horizontal, vertical and ``diagonal'' ($x+y$ or $x-y$ is a constant) lines to considered curves.
For webs $3$ and $4$ the images of both diagonal direction are remarkable curve: Apollonian circles and lines (vertical or horizontal).

Before going to the proof, let us say how these pencils relate with each other in an algebraic sense.
All circles in the plane can be considered as conics passing through two fixed points of the complex infinite line.
These points are called \emph{circular} and have homogeneous coordinates $I_1=(1,i,0)$ and $I_2=(1,-i,0)$.

Foci $F_1$ and $F_2$ of any conic $\alpha$ can be described as the intersection of tangent lines to $\alpha$ from the circular points.
The other pair of points of the intersection, denote it by $F_1'$ and $F_2'$, can also be considered as foci. 
On Figure~\ref{fig:projective description} the projective picture is drawn. 
Through each point $P$ in the plane pass the four conics: two conics touching sides of the quadrilateral $F_1F_1'F_2F_2'$ and two passing through its opposite vertices and $I_1$ and $I_2$.
The lines appearing in webs $3$ and $4$ can be described as lines passing through intersection of $I_1I_2$ and $F_1F_2$ (horizontal lines) and through the intersection of $I_1I_2$ and $F_1'F_2'$ (vertical lines).
Some triples of these conics have another joint intersections which correspond to points symmetric to $P$ in the axes (Figure~\ref{fig: four curves}). So, in general we describe webs formed by the family of conics tangent to four fixed lines and passing though points of their intersection.

\begin{center}
	\parbox{8cm}{
		\begin{center}
			\includegraphics{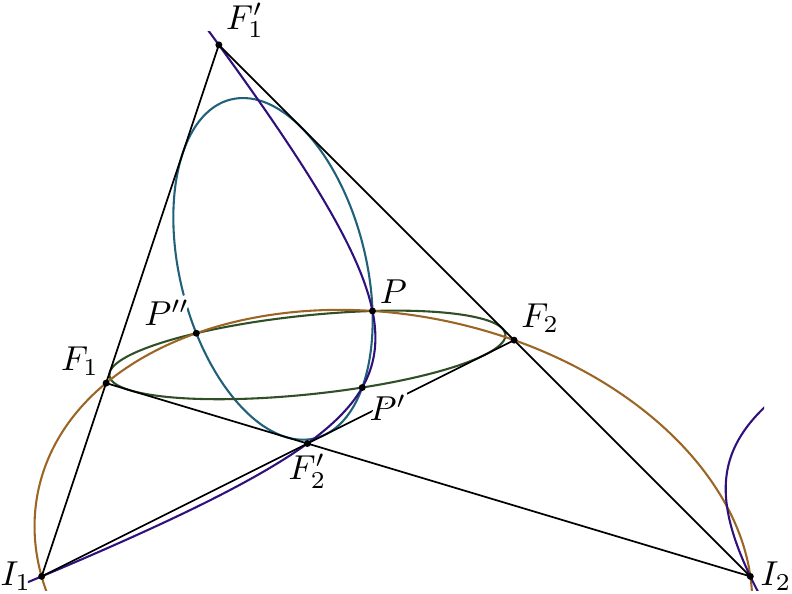}\\
			\f \label{fig:projective description}
		\end{center}
			}
	\parbox{7cm}{
		\begin{center}
			\includegraphics{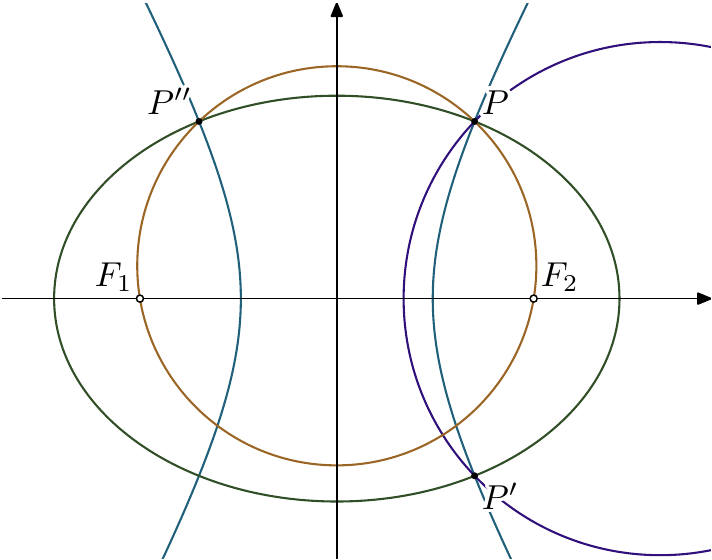}\\
			\f \label{fig: four curves}
		\end{center}
			}
\end{center}

W.~B\"ohm in \cite{bohm1970verwandte} constructed a net consisting of lines touching a conic. 
Quadrilaterals formed by lines of this net can be circumscribed around circles and points of intersection of these lines can be split into families lying on confocal conics.
This construction was rediscovered and generalized by the author and A. Bobenko in \cite{akopyan2016bobenko}, where also it was noticed that B\"ohm's net is a special case of the Poncelet grid introduced and investigated by R.~Schwarz \cite{schwartz2007poncelet}, see also~\cite{levi2007poncelet} for an additional discussion.

\begin{center}
	\parbox{8cm}{
		\begin{center}
			\includegraphics{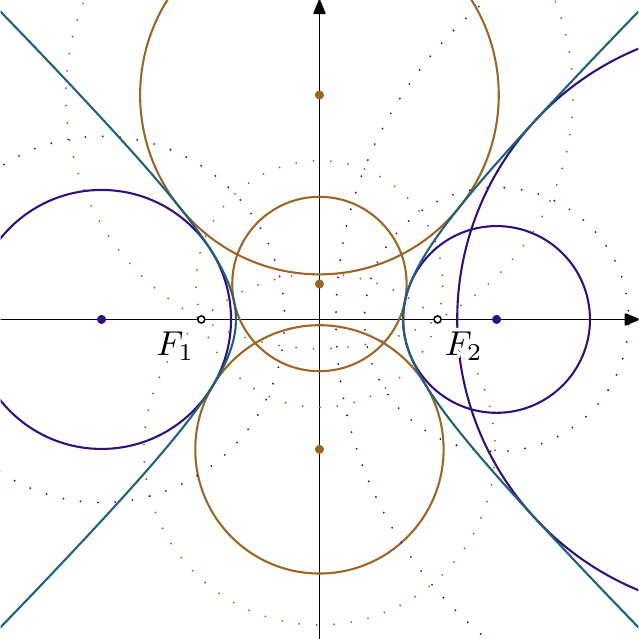}\\
			\f \label{fig:edelsbrunner construction}
		\end{center}
			}
	\parbox{7cm}{
		\begin{center}
			\includegraphics{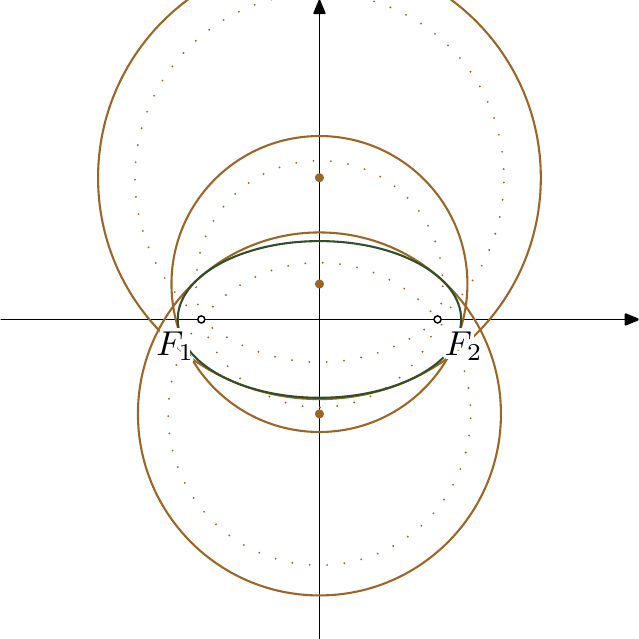}\\
			\f \label{fig:edelsbrunner construction ellipse}
		\end{center}
			}
\end{center}

The current constructions were inspired by the work of H.~Edelsbrunner~\cite{edelsbrunner1999deformable, cheng2001dynamic}, where he invented a new approach for designing smooth surfaces from a set of spheres. 
His method of connecting spheres by circumscribed hyperboloids is based on the following observation.
Let $k_1$, $k_2$ satisfy $k_1^2+k_2^2=1$.
If we scale each elliptic Apollonian circle of points $F_1$ and $F_2$  $k_1$ times and each hyperbolic Apollonian circle $k_2$ with respect to its center, the obtained circles touch a fixed hyperbola with foci $F_1$ and $F_2$.
Figure~\ref{fig:edelsbrunner construction} shows the case $k_1=k_2=\frac{1}{\sqrt 2}$, when the hyperbola is equilateral. 
Note that if $k_1>1$, then scaled elliptic Apollonian circles touch an ellipse with foci $F_1$ and~$F_2$ (Figure~\ref{fig:edelsbrunner construction ellipse}). (The touching points may have complex coordinates.)

Let us fix the notation we will use.
We suppose that foci of pencils have coordinates $F_1= (-1, 0)$ and $F_2=(1, 0)$.
Denote by $a(P)$ and $b(P)$ the distances from point $P$ to foci $F_1$ and $F_2$.
% Each point $P$ in the positive quadrant $\mathbb{R}^2_+$ uniquely determined by values $a(P)$ and $b(P)$ as the intersection of circles $\omega(F_1, a(P))$ and $\omega(F_2, b(P))$ and having positive $y$-coordinate.
Define the following functions:
\begin{itemize}
\item $f(P)=\displaystyle \frac{a(P)}{b(P)}$. 
The locus of points $f(P)=\const$, is a hyperbolic Apollonian circle corresponding to the points $F_1$ and $F_2$. For points $P\in \mathbb{R}^2_+$ we have $f(P)>1$.

\item $g(P)=\displaystyle \frac{a(P)^2+b(P)^2-4}{2a(P)b(P)}$. 
The locus of points $g(P)=\const \in (0, 1)$, $P\in \mathbb{R}^2_+$, is the upper arc of a circle passing through $F_1$ and $F_2$ (and the arc symmetric to it in $x$-axis). Indeed, for any point $P$ on this arc we have $\cos \angle F_1PF_2=g(P)=\const$.

\item $h(P)=a(P)-b(P)$.
The locus of points $h(P)=\const$ is a branch of a hyperbola with foci $F_1$ and $F_2$.
Note that for $P \in \mathbb{R}^2_+$ we have $0<h(P)<2$.

\item $e(P)=a(P)+b(P)$.
The locus of points $e(P)=\const$ is an ellipse with foci $F_1$ and~$F_2$.
\end{itemize}

Note that $P\in \mathbb{R}^2_+$ is uniquely determined by values of any two function from this list.
For simplicity in the following sections we denote the values of considered functions at point $P$ by $a$, $b$, $f$, $g$, $h$, and $e$.

\section{The web from Apollonian circles and confocal hyperbolas}

In this section we prove that two pencils of Apollonian circles with foci at $F_1$ and $F_2$, and the family of hyperbolas with foci at $F_1$ and $F_2$ form a trivial web in the positive quadrant $\mathbb{R}_+^2$.
% For simplicity here and in further section we construct map from $\Omega'$ to our domain (not like it was in the definition of the web).
Applying the map to vertices of a shifted lattice $k\mathbb{Z}^2+t$ we obtain a configuration shown on~Figure~\ref{fig:cir-cir-hyp}.

\begin{theorem}
	\label{thm:cir-cir-hyp}
	Suppose $\Omega$ be a set of points $\mathbb{R}^2$ with positive $x$-coordinate, and the map $\varphi: \Omega \rightarrow \mathbb{R}_+^2$ is defined in the following way:
	\begin{equation*}
		\label{eq:f_st, cir-cir-hyp}
		\varphi(x,y)=
		\left(
		\frac{
		{\rm e}^{x}\sqrt{1+{\rm e}^y}
		}{
		{\rm e}^{x}+{\rm e}^y},
		\frac{{\rm e}^y\sqrt{{\rm e}^{x}-1}}{{\rm e}^{x}+{\rm e}^y}
		\right).
	\end{equation*}
	Then $\varphi$ is surjective and maps lines $x=\const$ to arcs of circles of the elliptic Apollonian pencil, rays $y=\const$ to arcs of circles of the hyperbolic Apollonian pencil and rays $x-y=\const$ to arcs of hyperbolas with foci $F_1$ and $F_2$.
\end{theorem}

\begin{proof}
	For the proof we need to show that the value of $f$ depends only on $y$, $g$ only on~$x$, and $h$ only on $x-y$.
	We start the calculation from finding values of $a$ and $b$.
	Suppose the point $P$ has coordinates $(s,t)$ and it is an image of the point $(x,y)$ under the map $\varphi$.
	Then $a^2=(s+1)^2+t^2$ and $b^2=(s-1)^2+t^2$. Therefore
	\begin{multline*}
	a^2=\frac{
	({\rm e}^{x}+{\rm e}^y+{\rm e}^{x}\sqrt{1+{\rm e}^y})^2+{\rm e}^{x+2y}-{\rm e}^{2y}
	}
	{
	({\rm e}^{x}+{\rm e}^y)^2
	} =\\
	\frac{
		({\rm e}^{x}+{\rm e}^y)({\rm e}^{x}+{\rm e}^y+2{\rm e}^{x}\sqrt{1+{\rm e}^y})+{\rm e}^{2x} +{\rm e}^{2x+y}+{\rm e}^{x+2y}-{\rm e}^{2y}
	}
	{
		({\rm e}^{x}+{\rm e}^y)^2
	} =\\
	\frac{
		{\rm e}^{x}({\rm e}^{x}+{\rm e}^y)(1+2\sqrt{1+{\rm e}^y})+
		({\rm e}^{x}+{\rm e}^{y}){\rm e}^{y}+({\rm e}^{x}+{\rm e}^y)({\rm e}^{x+y}+{\rm e}^{x} -{\rm e}^y)
	}
	{
		({\rm e}^{x}+{\rm e}^y)^2
	} =\\
	\frac{
	{\rm e}^{x}
	}
	{
	({\rm e}^{x}+{\rm e}^y)
	}
	\left(1+2\sqrt{1+{\rm e}^y}+1+{\rm e}^y\right)
	=
	\frac{
		{\rm e}^{x}
	}
	{
		{\rm e}^{x}+{\rm e}^y
	}
	\left(\sqrt{1+{\rm e}^y} +1\right)^2.
	\end{multline*}
	
	Analogously
	\begin{equation*}
	b=\sqrt{
		\frac{
			{\rm e}^{x}
		}
		{
			{\rm e}^{x}+{\rm e}^y
		}
	}
	\left(\sqrt{1+{\rm e}^y} -1\right).
	\end{equation*}

	Now let us compute $f$, $g$, and $h$.
	
	\begin{equation*}
	f=\frac{a}{b}=\frac{\sqrt{1+{\rm e}^y} +1}{\sqrt{1+{\rm e}^y} -1}.
	\end{equation*}
	\begin{multline*}
	g=\frac{a^2+b^2-4}{2ab}= 
	\frac{
	\frac{
		{\rm e}^{x}
	}
	{
		{\rm e}^{x}+{\rm e}^y
	}
		\left(
		(\sqrt{1+{\rm e}^y} +1)^2+(\sqrt{1+{\rm e}^y} -1)^2
		\right)-4
	}
	{
	2\frac{
		{\rm e}^{x}
	}
	{
		{\rm e}^{x}+{\rm e}^y
	}
		\left(
		(\sqrt{1+{\rm e}^y}+1)(\sqrt{1+{\rm e}^y}-1)
		\right)
	}=\\
	\frac{
		{\rm e}^{x}(2{\rm e}^y+4)-4({\rm e}^{x}+{\rm e}^y)
	}{
	2{\rm e}^{x} {\rm e}^{y}
	}=1-2{\rm e}^{-x}.
	\end{multline*}
	
	\begin{equation*}
	h=a-b=\sqrt{
	\frac{
		{\rm e}^{x}
	}
	{
		{\rm e}^{x}+{\rm e}^y
	}
	}
	\left(
	(\sqrt{1+{\rm e}^y} +1)-(\sqrt{1+{\rm e}^y} -1)
	\right)=
	2\sqrt{\frac{
	{\rm e}^{x-y}
	}
	{
	{\rm e}^{x-y}+1
	}}.
	\end{equation*}
	
	The surjectivity of the map easily follows from the fact that varying $x$ and $y$ we can obtain all circles from Apollonian pencils.
\end{proof}

\section{The web from Apollonian circles and confocal ellipses}

We prove that two pencils of Apollonian circles with foci at $F_1$ and $F_2$, and the family of ellipses with foci at $F_1$ and $F_2$ form a trivial web in the positive quadrant $\mathbb{R}_+^2$ (Figure~\ref{fig:cir-cir-el}).
\begin{theorem}
	\label{thm:cir-cir-el}
	Suppose $\Omega$ is the forth quadrant, the set of points in the plane with positive $x$-coordinate and negative $y$-coordinate. The map $\varphi: \Omega \rightarrow \mathbb{R}_+^2$ is defined in the following way:
	\begin{equation*}
		\label{eq:v_st, cir-cir-el}
	\varphi(x,y)=
	\left(
	\frac{{\rm e}^{x}\sqrt{1-{\rm e}^{y}}}{{\rm e}^{x}-{\rm e}^{y}},
	\frac{{\rm e}^{y}\sqrt{{\rm e}^{x}-1}}{{\rm e}^{x}-{\rm e}^{y}}
	\right).
	\end{equation*}
	Then $\varphi$ is surjective and maps rays $x=\const$ to arcs of circles of the elliptic Apollonian pencil, rays $y=\const$ to arcs of circles of the hyperbolic Apollonian pencil and rays $x-y=\const$ to arcs of ellipses with foci $F_1$ and $F_2$.
\end{theorem}

\begin{proof}
	For the proof we need to show that the value of $f$ depends only on $y$, $g$ only on $x$, and $e$ only on $x-y$.
	As in the proof of Theorem~\ref{thm:cir-cir-hyp}, we start from the calculation of $a$ and $b$.
	\begin{multline*}
	a^2=\frac{
	({\rm e}^{x}-{\rm e}^{y}+{\rm e}^{x}\sqrt{1-{\rm e}^{y}})^2+{\rm e}^{2y+x}-{\rm e}^{2y}
	}
	{
	({\rm e}^{x+y}-{\rm e}^{y})^2
	} =\\
	\frac{
		({\rm e}^{x}-{\rm e}^{y})({\rm e}^{x}-{\rm e}^{y}+2{\rm e}^{x}\sqrt{1-{\rm e}^{y}})+
		{\rm e}^{2x} - {\rm e}^{2x+y}+{\rm e}^{x+2y}-{\rm e}^{2y}
	}
	{
	({\rm e}^{x}-{\rm e}^{y})^2
	} =\\
	\frac{
		{\rm e}^{x}({\rm e}^{x}-{\rm e}^{y})(1+2\sqrt{1-{\rm e}^{y}})-({\rm e}^{x}-{\rm e}^{y}){\rm e}^{y}  +({\rm e}^{x}-{\rm e}^{y})({\rm e}^{x} -{\rm e}^{x+y}+{\rm e}^{y})
	}
	{
	({\rm e}^{x}-{\rm e}^{y})^2
	} =\\
	\frac{
	{\rm e}^{x}
	}
	{
	{\rm e}^{x}-{\rm e}^{y}
	}
	\left(1+2\sqrt{1-{\rm e}^{y}}+1-{\rm e}^{y}\right)
	=
	\frac{
	{\rm e}^{x}
	}
	{
	{\rm e}^{x}-{\rm e}^{y}
	}
	\left(1+\sqrt{1-{\rm e}^{y}}  \right)^2.
	\end{multline*}
	
	Analogously
	\begin{equation*}
	b=\sqrt{\frac{
	{\rm e}^{x}
	}
	{
	{\rm e}^{x}-{\rm e}^{y}
	}}
	\left(1-\sqrt{1-{\rm e}^{y}}\right).
	\end{equation*}

	Now let us compute $f$, $g$ and $e$.
	
	\begin{equation*}
	f=\frac{a}{b}=\frac{1+\sqrt{1-{\rm e}^{y}}}{1-\sqrt{1-{\rm e}^{-y}} }.
	\end{equation*}
	\begin{multline*}
	g=\frac{a^2+b^2-4}{2ab}= 
	\frac{
		\frac{
				{\rm e}^{x}
			}
			{
				{\rm e}^{x}-{\rm e}^{y}
			}
		\left(
		(1+\sqrt{1-{\rm e}^{y}} )^2+(1-\sqrt{1-{\rm e}^{y}})^2
		\right)-4
	}
	{
	2\frac{
		{\rm e}^{x}
		}
		{
		{\rm e}^{x}-{\rm e}^{y}
		}
		\left(
		(1+\sqrt{1-{\rm e}^{y}} )(1-\sqrt{1-{\rm e}^{y}})
		\right)
	}=\\
	\frac{
		{\rm e}^{x}(4-2{\rm e}^{y})-4({\rm e}^{x}-{\rm e}^{y})
	}{
	2{\rm e}^{x} {\rm e}^{y}
	}=2{\rm e}^{-x}-1.
	\end{multline*}
	
	\begin{equation*}
	e=a+b=
	\sqrt{
		\frac{
			{\rm e}^{x}
		}
		{
			{\rm e}^{x}-{\rm e}^{y}
		}
	}
	\left(
	(1+\sqrt{{\rm e}^y+1})+(1-\sqrt{{\rm e}^y+1})
	\right)=
	2\sqrt{
		\frac{
			{\rm e}^{x-y}
		}
		{
			{\rm e}^{x-y}-1.
		}
	}
	\end{equation*}
	
	Again, the surjectivity of the map easily follows from the fact that varying $x$ and $y$ we can obtain all circles from Apollonian pencils.
\end{proof}

\section{The web from confocal conics and a hyperbolic Apollonian pencil}
In this section we prove that there is a trivial web formed by confocal ellipses and hyperbolas, and the hyperbolic Apollonian pencil.
The family of vertical lines can be adjoint to this web in a very natural way (Figure~\ref{fig:hip-el-ap}).

\begin{theorem}
	\label{thm:hip-el-ap}
	Suppose $\Omega$ is the forth quadrant and the map $\varphi: \Omega \rightarrow \mathbb{R}_+^2$ is defined in the following way:
	\begin{equation*}
		\label{eq:f_st, hip-el-ap}
		\varphi(x,y)=
		\left(
		{\rm e}^{x+y},
		\sqrt{
		({\rm e}^{2x}-1)(1-{\rm e}^{2y})
		}
		\right).
	\end{equation*}
	Then $\varphi$ is surjective and maps rays $x=\const$ and $y=\const$ to arcs of ellipses and hyperbolas with foci $F_1$ and $F_2$, rays $x-y=\const$ to arcs of circles of the hyperbolic Apollonian pencil with foci $F_1$ and $F_2$, and intervals $x+y=\const$ to vertical lines.
\end{theorem}

\begin{proof}
	In this case $a$ and $b$ have a very nice formulas:
	\begin{equation*}
	a=\sqrt{(1+2{\rm e}^{x+y}+{\rm e}^{2x+2y})+({\rm e}^{2x}-{\rm e}^{2x+2y}+{\rm e}^{2y}-1)}=({\rm e}^{x}+{\rm e}^{y}),
	\end{equation*}
	\begin{equation*}
	b=\sqrt{(1-2{\rm e}^{x+y}+{\rm e}^{2x+2y})+({\rm e}^{2x}-{\rm e}^{2x+2y}+{\rm e}^{2y}-1)}=({\rm e}^{x}-{\rm e}^{y}).
	\end{equation*}
%
% The first net formed by confocal ellipses and hyperbolas.
% Let $s=\log(e/2)$. Since $e$ always greater than $2$ we have that $s$ attain all values of $\mathbb{R}_+$.
%
% Let $t=-\log(h/2)$. Since $h$ is always less than $2$ we have that $t$ attain all values of $\mathbb{R}_+$.

% Let us calculate $a$, $b$ and coordinates of the point $P(s,t)$.
% \begin{equation*}
% a=\frac{1}{2}\left(e+h\right)={\rm e}^s+{\rm e}^{-t}, \hskip 1cm
% b=\frac{1}{2}\left(e-h\right)={\rm e}^s-{\rm e}^{-t},
% \end{equation*}

We see that the values $e=a+b=2{\rm e}^x$ and $h=a-b=2{\rm e}^{y}$ defined only by $x$ and $y$, respectively, it is left to check that the ratio $a/b$ depends only on $x-y$:
% , therefore the images of lines $x+y=\const$ is an arc of a hyperbolic Appolonian circle:
\begin{equation*}
f=\frac{a}{b} = 
\frac{{\rm e}^x+{\rm e}^{y}}{{\rm e}^x-{\rm e}^{y}} =
\frac{{\rm e}^{\frac{x-y}{2} }+{\rm e}^{\frac{y-x}{2}}}{{\rm e}^{\frac{x-y}{2} }-{\rm e}^{\frac{{y-x}}{2}}} = \coth \left(\frac{x-y}{2} \right).
\end{equation*}

The fact that image of points with $x+y=\const$ is a vertical lines follows directly form the definition of $\varphi$.
Since varying $x$ and $y$ we can obtain all ellipses and hyperbolas from the confocal family the image of $\varphi$ is the whole quadrangle $\mathbb{R}_+^2$.
\end{proof}

\section{The web from confocal conics and an elliptic Apollonian pencil}
The last example of a trivial web is formed by confocal ellipses and hyperbolas, and and elliptic Apollonian pencil, and family of horisontal lines, which can be adjoint to the web as well (Figure~\ref{fig:hip-ell-pas}).

\begin{theorem}
	\label{thm:hip-ell-pas}
	Suppose $\Omega$ is the forth quadrant and the map $\varphi: \Omega \rightarrow \mathbb{R}_+^2$ is defined in the following way:
		\begin{equation*}
		\label{eq:f_st, hip-ell-pas}
		\varphi(x,y)=
		\left(
		\sqrt{(1+{\rm e}^x)(1-{\rm e}^{y})},
		\sqrt{{\rm e}^{x+y}}
		\right).
	\end{equation*}
	Then $\varphi$ is surjective and maps rays $x=\const$ and $y=\const$ to arcs of ellipses and hyperbolas with foci $F_1$ and $F_2$, rays $x-y=\const$ to arcs of circles of the elliptic Apollonian pencil with foci $F_1$ and $F_2$, and intervals $x+y=\const$ to horizontal lines.
\end{theorem}

\begin{proof}
	Again, find the formulas for $a$ and $b$.
	\begin{multline*}
	a^2=(\sqrt{(1+{\rm e}^x)(1-{\rm e}^{y})}+1)^2+{\rm e}^{x+y}=\\
	1+2\sqrt{(1+{\rm e}^x)(1-{\rm e}^{y})}+1+{\rm e}^x-{\rm e}^{y}=
	\left(\sqrt{1+{\rm e}^x}+\sqrt{1-{\rm e}^{y}}\right)^2.
	\end{multline*}
	Analogously 
	\begin{multline*}
	b^2=(\sqrt{(1+{\rm e}^x)(1-{\rm e}^{y})}-1)^2+{\rm e}^{x+y}=\\
	1-2\sqrt{(1+{\rm e}^x)(1-{\rm e}^{y})}+1+{\rm e}^x-{\rm e}^{y}=
	\left(\sqrt{1+{\rm e}^x}-\sqrt{1-{\rm e}^{y}}\right)^2.
	\end{multline*}
	
	We have $e=a+b=2\sqrt{1+{\rm e}^x}$ and $h=a-b=2\sqrt{1+{\rm e}^{y}}$.
	Let us calculate the value of~$g$.
	\begin{equation*}
	g=\frac{
		(\sqrt{1+{\rm e}^x}+\sqrt{1-{\rm e}^{y}})^2+(\sqrt{1+{\rm e}^x}-\sqrt{1-{\rm e}^{y}})^2-4
	}{
		2(\sqrt{1+{\rm e}^x}+\sqrt{1-{\rm e}^{y}})(\sqrt{1+{\rm e}^x}-\sqrt{1-{\rm e}^{y}})
	} =\\
	\frac{
	{\rm e}^x-{\rm e}^{y}
	}{
	{\rm e}^x+{\rm e}^{y}
	}=
	\tanh \left( \frac{x-y}{2} \right).
	\end{equation*}
	We see, that it depends only on $x-y$. 
	It is left to notice that from the definition of $\varphi$, follows that images of intervals $x+y=const$ are horizontal lines.
	The map is surjective by the same arguments as in the previous theorem.
	
	%
	% \[
	% x=\frac{a^2-b^2}{4}=
	% \frac{
	% (\sqrt{1+{\rm e}^s}+\sqrt{1-{\rm e}^{-t}})^2-(\sqrt{1+{\rm e}^s}-\sqrt{1-{\rm e}^{-t}})^2
	% }{4}=\sqrt{(1+{\rm e}^s)(1-{\rm e}^{-t})}.
	% \]
	%
	% \begin{multline*}
	% y^2=\frac{a^2+b^2}{2}-1-x^2=\\
	% \frac{(\sqrt{1+{\rm e}^s}+\sqrt{1-{\rm e}^{-t}})^2+(\sqrt{1+{\rm e}^s}-\sqrt{1-{\rm e}^{-t}})^2}{2}
	% -1-(1+{\rm e}^s)(1-{\rm e}^{-t})=\\
	% (1+{\rm e}^s)+(1-{\rm e}^{-t})-1-(1+{\rm e}^s-{\rm e}^{-t}-{\rm e}^{s-t})={\rm e}^{s-t}.
	% \end{multline*}
\end{proof}

% First net formed by confocal ellipses and hyperbolas.
% Let $s=\log(\frac{e^2-4}{4})$. Again $s$ attains all values of $\mathbb{R}^+$ and $e=2 \sqrt{1+{\rm e}^s}$.
%
% Let $t=\log(\frac{4}{4-h^2})$. Then $h=2 \sqrt{1-{\rm e}^{-t}}$. Since $h$ is always less than $2$ we have that $s$ attain all values of $\mathbb{R}_+$.

% Now we can find formula for the points of the net.
% \begin{equation*}
% 	\label{eq:v_st, hip-ell-pas}
% 	P(s,t)=
% 	\left(
% 	\sqrt{(1+{\rm e}^s)(1-{\rm e}^{-t})},
% 	\sqrt{{\rm e}^{s-t}}
% 	\right).
% \end{equation*}

% > g(m,n):= sqrt(exp(m)+1);
% (m, n) -> sqrt(exp(m) + 1)
% >
% > h(m,n):= sqrt(-exp(- n)+1);

\begin{center}
	\includegraphics{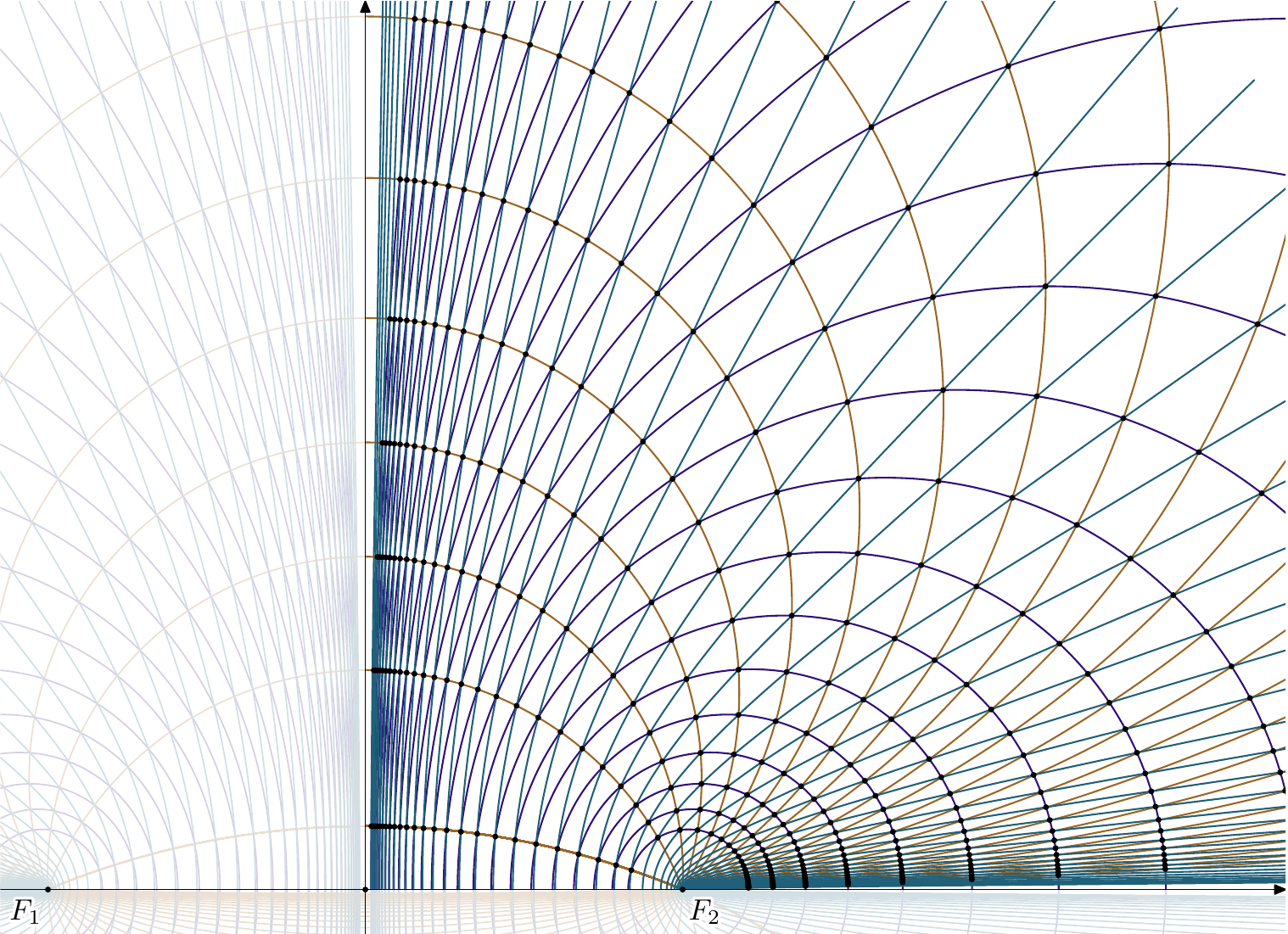}\\
	\f \label{fig:cir-cir-hyp}
\end{center}

\begin{center}
\includegraphics{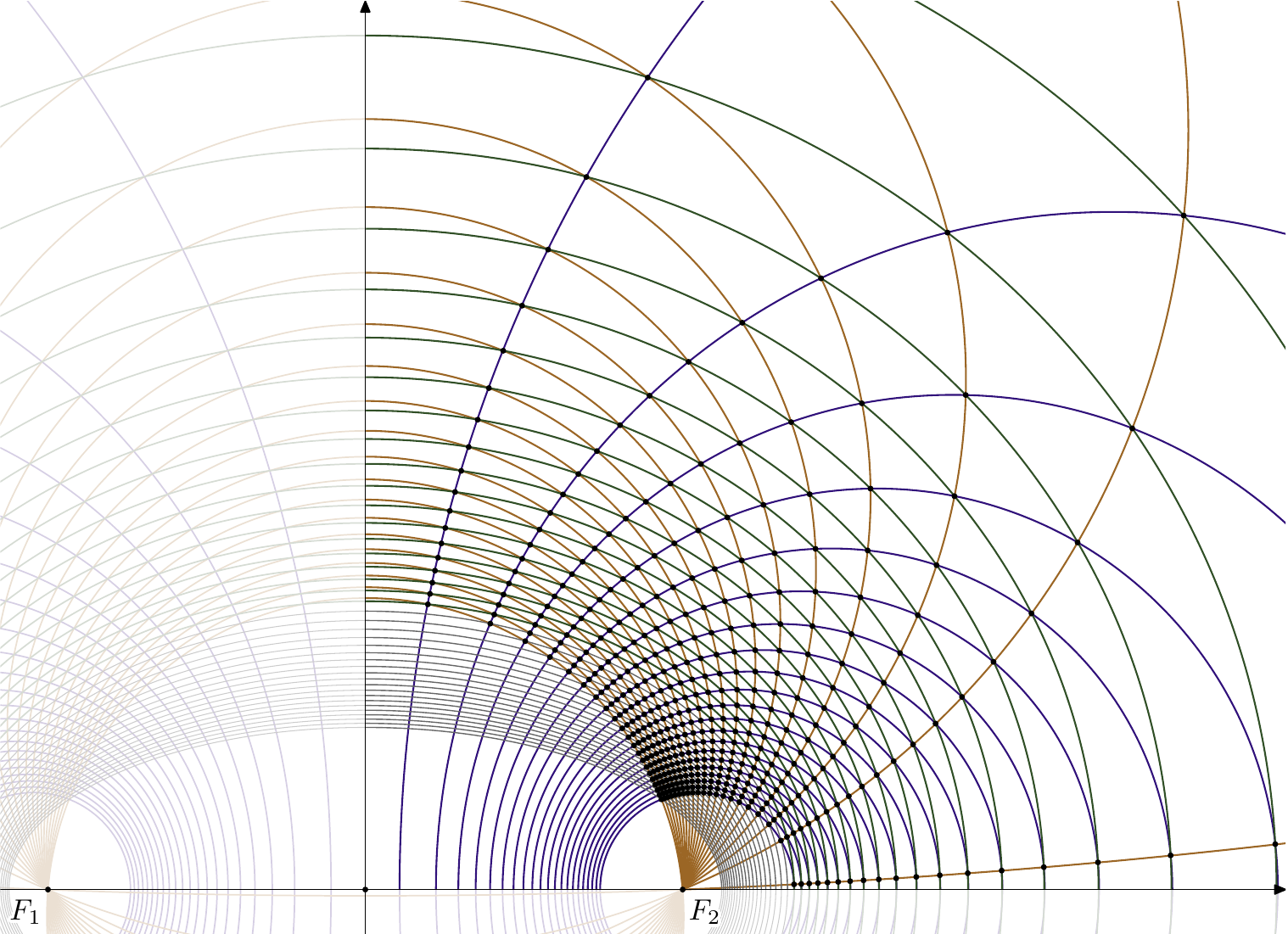}

\f \label{fig:cir-cir-el}
\end{center}

\begin{center}
	\includegraphics{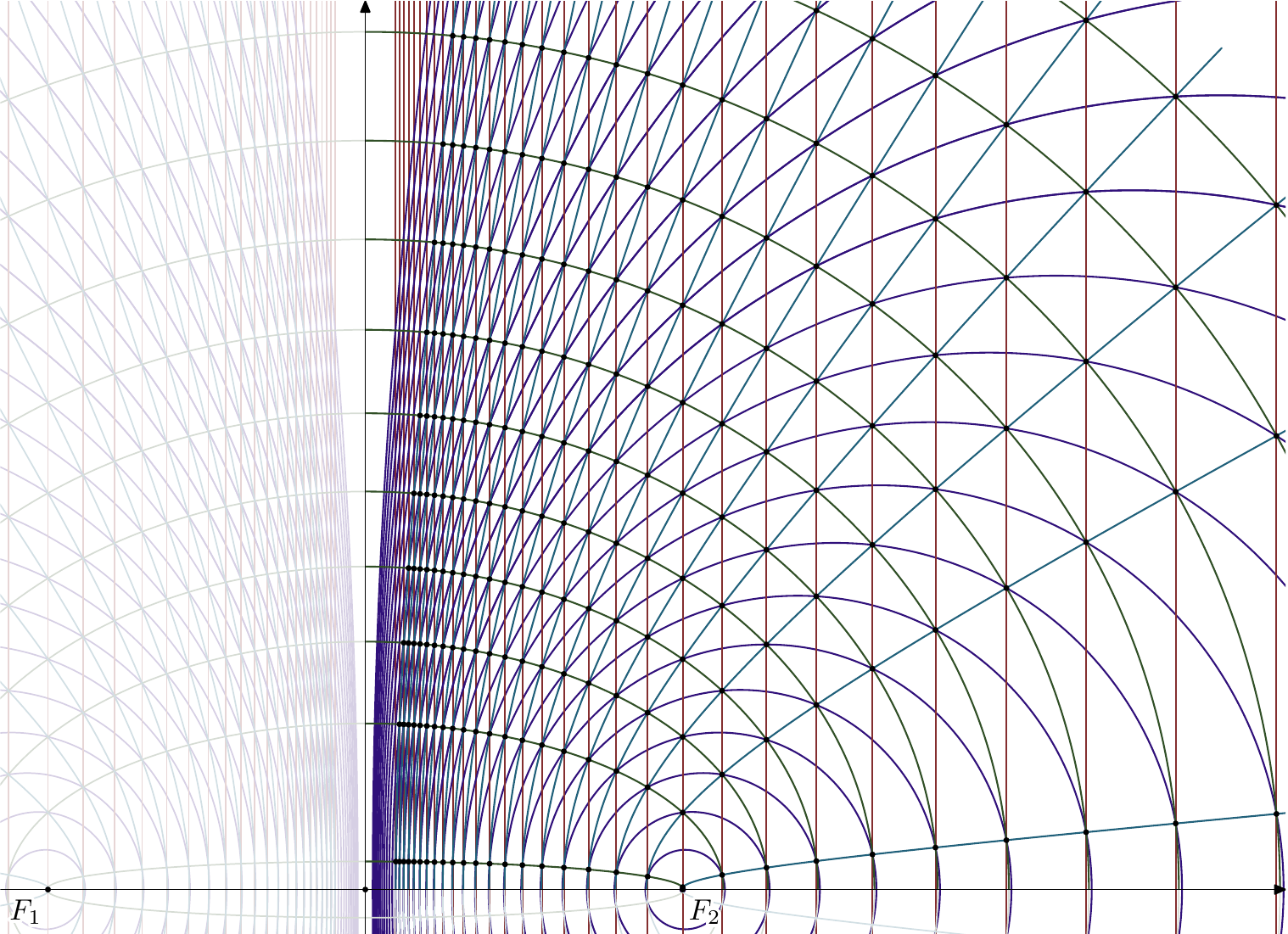}\\
	\f \label{fig:hip-el-ap}
\end{center}

\begin{center}
	\includegraphics{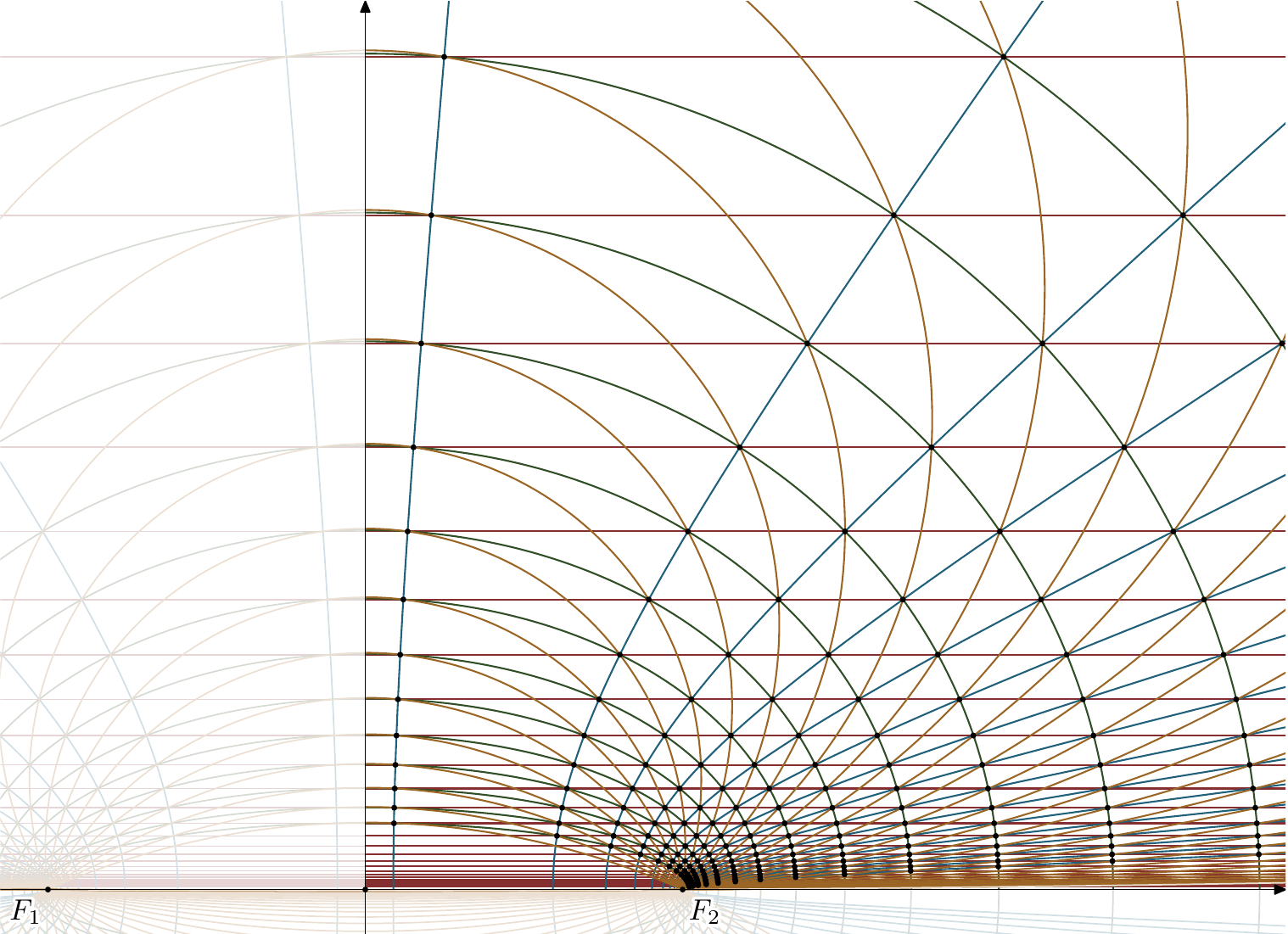}\\
	\f \label{fig:hip-ell-pas}
\end{center}


\begin{thebibliography}{10}

\bibitem{akopyan2016bobenko}
A.~Akopyan and A.~Bobenko.
\newblock Incircular nets and confocal conics.
\newblock 06 2016, 
\newblock ArXiv:1602.04637.

\bibitem{akopyan2007geometry}
A.~V. Akopyan and A.~A. Zaslavsky.
\newblock {\em Geometry of conics}, volume~26 of {\em Mathematical World}.
\newblock American Mathematical Society, Providence, RI, 2007.
\newblock Translated from the 2007 Russian original by Alex Martsinkovsky.

\bibitem{blaschke1955einfuhrung}
W.~Blaschke.
\newblock {\em Einf\"uhrung in die {G}eometrie der {W}aben}.
\newblock Birkh\"auser Verlag, Basel und Stuttgart, 1955.

\bibitem{bohm1970verwandte}
W.~B{\"o}hm.
\newblock Verwandte {S}\"atze \"uber {K}reisvierseitnetze.
\newblock {\em Arch. Math. (Basel)}, 21:326--330, 1970.

\bibitem{cheng2001dynamic}
H.-L. Cheng, T.~K. Dey, H.~Edelsbrunner, and J.~Sullivan.
\newblock Dynamic skin triangulation.
\newblock {\em Discrete Comput. Geom.}, 25(4):525--568, 2001.
\newblock The Micha Sharir birthday issue.

\bibitem{edelsbrunner1999deformable}
H.~Edelsbrunner.
\newblock Deformable smooth surface design.
\newblock {\em Discrete Comput. Geom.}, 21(1):87--115, 1999.

\bibitem{fuchs2007mathematical}
D.~Fuchs and S.~Tabachnikov.
\newblock {\em Mathematical omnibus}.
\newblock American Mathematical Society, Providence, RI, 2007.
\newblock Thirty lectures on classic mathematics.

\bibitem{glaeser2016universe}
G.~Glaeser, H.~Stachel, and B.~Odehnal.
\newblock {\em The Universe of Conics: From the ancient Greeks to 21st century
  developments}.
\newblock Springer, 2016.

\bibitem{levi2007poncelet}
M.~Levi and S.~Tabachnikov.
\newblock The {P}oncelet grid and billiards in ellipses.
\newblock {\em The American Mathematical Monthly}, 114(10):895--908, 2007.

\bibitem{nilov2014onnewconstructions}
F.~K. Nilov.
\newblock On new constructions in the {B}laschke-{B}ol problem.
\newblock {\em Mat. Sb.}, 205(11):125--144, 2014.

\bibitem{schwartz2007poncelet}
R.~E. Schwartz.
\newblock The {P}oncelet grid.
\newblock {\em Advances in Geometry}, 7(2):157--175, 2007.

\end{thebibliography}
\end{document}